\documentclass[11pt]{amsart}
\usepackage[utf8]{inputenc}
\usepackage{amsmath}%
\usepackage{amsfonts}%
\usepackage{amssymb}%
\usepackage{amsthm}
\usepackage{bm}
\usepackage{graphicx}
\usepackage{stmaryrd}
\usepackage{centernot} 
\usepackage{tikz}
\usepackage{wrapfig}
\usepackage{mathtools}

\usepackage{url}
\usepackage{geometry}
\usepackage{enumerate}

\usepackage{latexsym}

\theoremstyle{plain}
\newtheorem{theorem}{Theorem}[section]
\newtheorem{corollary}[theorem]{Corollary}
\newtheorem{lemma}[theorem]{Lemma}

\newtheorem{claim}[theorem]{Claim}
\newtheorem{proposition}[theorem]{Proposition}
\newtheorem{problem}{Problem}

\theoremstyle{definition}
\newtheorem{definition}[theorem]{Definition}

\newtheorem{remark}[theorem]{Remark}

\newcommand{\mcV}{V}
\newcommand{\mcU}{U}
\newcommand{\mcW}{W}

\newcommand{\mcB}{\mathcal{B}}

\newcommand{\Xspace}{X}
\newcommand{\Gspace}{G}
\newcommand{\Yspace}{Y}

\newcommand{\Hspace}{H}
\newcommand{\BPi}{\mathbf{\Pi}}
\newcommand{\BSigma}{\mathbf{\Sigma}}

\newcommand{\wreath}{\text{ Wr }}

\begin{document}

\title[Obstructions to classification by (co)homology and other TSI invariants]{Dynamical obstructions to classification by (co)homology and other TSI-group invariants}

\author{Shaun Allison}
\address{Department of Mathematical Sciences, Carnegie Mellon University, Pittsburgh, PA 15213}
\email{sallison@andrew.cmu.edu}
\urladdr{https://www.math.cmu.edu/~sallison/}

\author{Aristotelis Panagiotopoulos}
\address{Mathematics Department, Caltech, 1200 E. California Blvd,
Pasadena, CA 91125}
\email{panagio@caltech.edu}
\urladdr{http://www.its.caltech.edu/~panagio/}

\thanks{We are grateful to A. Shani,  M. Lupini,  J. Bergfalk, and A.S. Kechris for all the useful and inspiring discussions, as well as to  S. Coskey  and J.D. Clemens for sharing  an early draft of \cite{ClemensCoskey} with us.
We would also like to thank the anonymous referee for their valuable comments and for raising our attention to several subtle errors in an earlier version of this paper.
Finally,  we want to acknowledge the hospitality and financial support of the California Institute of Technology during the visit of S.A. in the winter of 2020
}
      	


\begin{abstract}
In the spirit of Hjorth's turbulence theory, we introduce ``unbalancedness": a new dynamical obstruction to classifying orbit equivalence relations by actions of Polish groups  which  admit a two side invariant metric (TSI). Since abelian groups are TSI, unbalancedness can be used for identifying which classification problems cannot be solved  by classical homology and cohomology theories.

In terms of applications, we show that Morita equivalence of continuous-trace $C^*$-algebras, as well as isomorphism of Hermitian line bundles, are not classifiable by  actions of TSI groups. In the process, we show that the Wreath product of any two non-compact subgroups of $S_{\infty}$ admits 
an action whose orbit equivalence relation is generically ergodic against any action of a TSI group and we deduce that there is an orbit equivalence relation of a CLI group which is not classifiable by actions of TSI groups.

\end{abstract}

\maketitle

\section{Introduction}

\subsection{Classification by (co)homological invariants}
One of the leading questions in many mathematical research programs is  whether a certain classification problem admits a ``satisfactory" solution. What constitutes 
a satisfactory solution depends of course on the context, and it is often subject to change 
when the original goals are deemed hopeless. \emph{Invariant descriptive set theory} provides a formal framework for measuring the complexity of  classification problems and for showing which types of
 invariants are inadequate for complete classification.

 Hjorth's turbulence theory is one of the biggest accomplishments in this direction as it provides obstructions to  \emph{classification by countable structures}, i.e., for classification using only isomorphism types of countable structures as invariants. Historically, this type of classification played important role in ergodic theory since von Neumann's seminal paper 
 \cite{NeumannZurOI},  where he shows that ergodic measure-preserving transformations of  discrete spectrum  are completely classified up to isomorphism by their (countable) spectrum.  This led to the hope that, by adding further structure on these countable invariants, one could classify all ergodic measure-preserving transformations via countable structures. However, 70 years after the publication of \cite{NeumannZurOI}, the  newly developed theory of turbulence \cite{Hjorth2010} was used to establish that this hope was too optimistic; see \cite{Hjorth2001, Foreman2004AnAT}.

Another very common type of classification that one encounters in mathematical practice is  {\bf classification by (co)homological invariants}, i.e., classification using elements of homology or cohomology groups of some appropriate  chain complex. For example:
\begin{enumerate}
\item there is an assignment $p\mapsto c(p)$, from Hermitian line bundles  $p\colon E\to B$ over a locally compact metrizable space $B$,  to the  \emph{\v{C}ech cohomology group} $\mathrm{H}^2(B)$  of $B$, so that $p$ and $p'$ are isomorphic over $B$ iff $c(p)=c(p')$; see \cite{Morita}.
\item there is an assignment $\mathcal{A}\mapsto c(\mathcal{A})$, 
from continuous-trace $C^{*}$-algebras $\mathcal{A}$  with  locally compact metrizable spectrum $S$,  to  the  \emph{\v{C}ech cohomology group} $\mathrm{H}^3(S)$  of $S$, so that $\mathcal{A}$ and $\mathcal{A}'$ are Morita equivalent over $S$,  iff $c(\mathcal{A})=c(\mathcal{A}')$; see \cite{Blackadar}.
\item there is an assignment $e\mapsto c(e)$, from group extensions   $e\colon E \to \mathbb{Q}_{2}$ of the dyadic rationals $\mathbb{Q}_{2}$ by $\mathbb{Z}$, to the \emph{Steenrod homology group} $\mathrm{H}_0(\Sigma)$ of the character $\Sigma$ of $\mathbb{Q}_{2}$, so that $e$ and $e'$ are isomorphic iff  $c(e)=c(e')$; see \cite{eilenberg_group_1942}. 
\end{enumerate}
Up until recently,  classification by (co)homological invariants had not been considered from the perspective of invariant descriptive set theory. However, as a consequence of the results in \cite{BLP2019}, classification by the above (co)homological invariants  forms a well defined complexity class in the standard setup of invariant descriptive set theory. This complexity class  is entirely contained within the class of all classification problems which are  {\bf classifiable by abelian group actions}; see Problem \ref{Problem}. As a consequence,
Corollary \ref{Cor:Main} below provides a new anti-classification criterion,
in the spirit of Hjorth's turbulence theory,   which can be used as
an obstruction to classification by (co)homological invariants.
To illustrate its use, we will apply it to show that the \emph{coordinate free} versions of Hermitian line bundle isomorphism and of Morita equivalence between continuous-trace $C^{*}$-algebras cannot be classified by (co)homological invariants as in the examples (1) and (2) above.


\subsection{Classification problems}
The formal framework that is often used for measuring the complexity of  classification problems is the Borel reduction hierarchy.
Formally, a {\bf classification problem} is a pair $(\Xspace,E)$, where $\Xspace$ is a Polish space and $E$ is an analytic equivalence relation. A classification problem $(\Xspace,E)$ is considered to be of ``less or equal complexity" to the classification problem $(\Yspace,F)$, if there is a {\bf Borel reduction from $E$ to $F$}, i.e., a Borel map $f\colon \Xspace\to \Yspace$, so that for all $x,x'\in \Xspace$ we have: 
\[x E x' \iff f(x) F f(x').\]
At the lower end of this complexity hierarchy we have---in increasing complexity---the classification problems which are: concretely classifiable; essentially countable; and classifiable by countable structures.
A classification problem $(\Xspace,E)$ is  {\bf concrete classifiable} if $E$ Borel reduces to the equality relation of some Polish space. It is  {\bf essentially countable} if it Borel reduces to a Borel equivalence relation $F$ which has  equivalence classes of countable size. We finally say that $(\Xspace,E)$ is {\bf classifiable by countable structures} if $(\Xspace,E)$ Borel reduces to the problem $(\Xspace_{\mathcal{L}},\simeq_{\mathrm{iso}})$, of classifying, up to isomorphism, all countable $\mathcal{L}$-structures (graphs, groups, rings, etc.) of some fixed language $\mathcal{L}$. 

In order to show that two classification problems differ in complexity, it is imperative to have  a basic obstruction theory for Borel reductions.  These obstructions often come from dynamics and they directly apply to classification problems of the form $(\Xspace,E^\Gspace_\Xspace)$, where  $E^\Gspace_\Xspace$ is the  \emph{orbit equivalence relation} of the continuous action of a Polish group $\Gspace$ on a Polish space $\Xspace$.
A classical dynamical obstruction to concrete classification is \emph{generic ergodicity}; see \cite{Gao2008}. 
Similarly, Hjorth's \emph{turbulence theory} \cite{Hjorth2010} provides obstructions to classifiability by countable structures. Finally, \emph{storminess} \cite{Hjorth2005}, as well as \emph{local approximability} \cite{KMPZ2019}, are both dynamical obstruction to being essentially countable.

These dynamical obstructons can be seen to answer the following general problem that was considered in \cite{LupiniPanagio2018}: if  $\mathcal{C}$ is a class of Polish groups, then we say that $(\Xspace,E)$ is {\bf classifiable by $\mathcal{C}$-group actions} if it is Borel reducible to an orbit equivalence relation $(\Yspace,E^\Hspace_\Yspace)$, where $\Hspace$ is a group from $\mathcal{C}$. 
\begin{problem}\label{Problem}
Given a class $\mathcal{C}$ of Polish groups, which dynamical conditions on a Polish $\Gspace$-space $\Xspace$ ensure that  $(\Xspace,E^{\Gspace}_{\Xspace})$ is not classifiable by $\mathcal{C}$-group actions?
\end{problem}
Indeed, generic ergodicity provides an answer  for  $\mathcal{C}=\{\text{compact Polish groups}\}$;
turbulence for $\mathcal{C}=\{\text{non-Archimedean Polish groups}\}$; while storminess and local-approximability for the class $\mathcal{C}=\{\text{locally-compact Polish groups}\}$. More recently, an answer to this problem for the case where $\mathcal{C}$ is the class of all Polish CLI groups has been given \cite{LupiniPanagio2018}. Recall that a Polish group is {\bf CLI} if it admits a complete and left-invariant metric. 
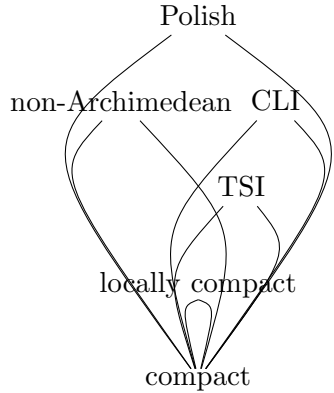
\begin{figure}[!htbp]
    \centering
\begin{tikzpicture}[scale=0.57]
\node (f) at (0, 10) {Polish};
\node (e) at (1.8, 8) {CLI};
\node (t) at (1, 6) {TSI};
\node (d) at (-1.8, 8) {non-Archimedean};
\node (c) at (0, 3.7) {locally compact};
\node (b) at (0, 1.5) {};
\node (b') at (0, 1.5) {compact};

\draw (b) .. controls (-0.4,3.2) .. (0, 3.35);
\draw (b) .. controls (0.4,3.2) .. (0, 3.35);
\draw (b) .. controls (-3.4,6.3) .. (d);
\draw (b) .. controls (1,5) .. (d);
\draw (b) .. controls (3.4,6.3) .. (e);
\draw (b) .. controls (-1,5) .. (e);
\draw (b) .. controls (2.2,4.5) .. (t);
\draw (b) .. controls (-0.8,4) .. (t);
\draw (b) .. controls (-4,7) .. (f);
\draw (b) .. controls (4,7) .. (f);
\end{tikzpicture}
    \caption{Classification by $\mathcal{C}$-group actions for various group classes $\mathcal{C}$.}
    \label{fig1}
\end{figure}

The main goal of this paper is to provide an answer to the above problem for the case where $\mathcal{C}$ is the class of all Polish groups which admit a two side invariant ({\bf TSI})  metric.  Since every abelian group is TSI the obstructions we introduce can be used for showing when a classification problem is not classifiable by (co)homological invariants.

\subsection{Definitions and main results} \label{Intro_2}

A {\bf Polish space} $\Xspace$ is a separable, completely metrizable topological space. A {\bf Polish group} $\Gspace$ is a topological group whose topology is Polish. Let $d$ be a metric on $\Gspace$ that is compatible with the topology. We say that $d$ is {\bf left invariant} if $d(gh,gh')=d(h,h')$, for all $g,h,h'\in \Gspace$ and {\bf right invariant} if $d(hg,h'g)=d(h,h')$, for all $g,h,h'\in \Gspace$. We say that $d$ is {\bf two side invariant}, if it is both left and right invariant. We say that a Polish group is {\bf TSI}, if it admits a two sided invariant metric which is compatible with the topology. Such groups are often called {\bf balanced} since, by a theorem of Klee, they are precisely the Polish groups which admit a neighborhood basis of the identity consisting of conjugation-invariant open sets. We say that $\Gspace$ is {\bf non-Archimedean}, if it admits a basis of open neighborhoods of the identity consisting of open subgroups.
A {\bf Polish $\Gspace$-space} is a Polish space $\Xspace$  together with a continuous left action of a Polish group $\Gspace$ on $\Xspace$. If $x\in \Xspace$, we write $[x]$ to denote the orbit $\Gspace  x$ of $x$. We  denote by $E^\Gspace_\Xspace$ the associated {\bf orbit equivalence relation}: 
\[x E^\Gspace_\Xspace y  \iff  [x]=[y].\]

\begin{definition}
Let $\Xspace$ be a Polish $\Gspace$-space and let $x,y\in\Xspace$. We write  $x \leftrightsquigarrow	 y$, if for every open neighborhood $\mcV$ of the identity of $\Gspace$ and every open set $U\subseteq X$ having nonempty intersection with the orbit of $x$ or $y$, there exist $g^x, g^{y}\in \Gspace$ with $g^x x\in U$ and $g^y y\in U$, so that:
\[(g^y y) \in \overline{\mcV  (g^x x)} \text{ and } (g^x x) \in \overline{\mcV  (g^y y)}.\]
\end{definition}
It is clear that $x \leftrightsquigarrow y$ implies that $\overline{Gx} = \overline{Gy}$. It is also clear that this is a symmetric relation that is invariant under the action of $\Gspace$; that is, if $g,h\in\Gspace$, we have that $x \leftrightsquigarrow y$ if and only if $gx \leftrightsquigarrow h y$. As a consequence we can write $[x] \leftrightsquigarrow [y]$ whenever $x \leftrightsquigarrow y$, without ambiguity.

\begin{definition}
Let $\Xspace$ be a Polish $\Gspace$-space and let $C\subseteq \Xspace$ be a $\Gspace$-invariant set. The {\bf unbalanced graph} associated to the action of $G$ on $C$  is the graph $(C/\Gspace,\leftrightsquigarrow)$, where $C/\Gspace:=\{[x]\mid x\in C\}$, and $[x]\leftrightsquigarrow[y]$ if and only if $x\leftrightsquigarrow y$.
\end{definition}

We say that $(C/\Gspace,\leftrightsquigarrow)$ is {\bf connected}, if for every $x,y\in C$ there is a {\bf path in $C/\Gspace$ from $[x]$ to $[y]$}; that is, a sequence $x_0,\ldots,x_{n-1}\in C$ so that $x=x_0, x_{n-1}=y$ and $x_{i-1}\leftrightsquigarrow x_i$, for all $0<i<n$. 
We say that $(\Xspace/\Gspace,\leftrightsquigarrow)$  is {\bf generically semi-connected}, if for every $\Gspace$-invariant comeager $C\subseteq \Xspace$, there is a comeager $D\subseteq C$ 
so that for every $x,y\in D$ there is a path between $[x]$ and $[y]$ in $(C/\Gspace,\leftrightsquigarrow)$. We say that a Polish $G$-space $X$ is {\bf generically unbalanced}, if it has meager orbits and $(\Xspace/\Gspace,\leftrightsquigarrow)$ is generically semi-connected.

Let $(\Xspace,E), (\Yspace,F)$ be two classification problems. A {\bf Baire-measurable homomorphism} from $E$ to $F$ is a a Baire-measurable map $f: \Xspace \rightarrow \Yspace$ so that $xEx'\implies f(x) F f(x')$. It is a {\bf Baire-measurable reduction}, if we additionally have $ f(x) F f(x')\implies xEx'$.  We say that $(\Xspace,E)$ is {\bf generically $F$-ergodic}, if for every Baire-measurable homomorphism from $E$ to $F$ there is a comeager subset $C\subseteq \Xspace$ so that $f(x) F (x')$ for all $x,x'\in C$. The following theorem and its corollary are the main results of this paper.

\begin{theorem}\label{Th:main}
 Let $\Xspace$ be a Polish $\Gspace$-space and let $\Yspace$ be  a Polish $\Hspace$-space, where $\Hspace$ is TSI.
If $(\Xspace/\Gspace,\leftrightsquigarrow)$ is   generically semi-connected, then $E^\Gspace_\Xspace$ is generically $E^\Hspace_\Yspace$-ergodic. 
 \end{theorem}
 
 \begin{corollary}[Obstruction to classification by TSI]\label{Cor:Main}
If the Polish  $G$-space $X$ is  generically unbalanced, then the orbit equivalence relation $E^\Gspace_\Xspace$ is not classifiable by TSI-group actions. 
 \end{corollary}{}

We now turn  to applications. In \cite{ClemensCoskey}, a new family of jump operators $E\mapsto E^{[\Gamma]}$ was introduced which are similar to the Friedman–Stanley jump $E\mapsto E^{+}$: for every countable group $\Gamma$,  the {\bf $\Gamma$-jump} of the classification problem $(\Xspace,E)$ is the classification problem   
\[(\Xspace^{\Gamma},E^{[\Gamma]}), \quad \text{with}\quad x E^{[{\Gamma}]} x' \iff (\exists \gamma \in \Gamma) \; (\forall \alpha \in \Gamma) \; x(\gamma^{-1}\alpha) \; E  \; x'(\alpha).\]

In the same paper they showed that the $\mathbb{Z}$-jump $E^{[\mathbb{Z}]}_0$ of $E_0$ is generically ergodic with respect to the countable product $E^{\omega}_0$ of $E_0$. 
In \cite{Allison2020}, the stronger result was shown that, in fact, $E^{[\mathbb{Z}]}_0$ is  generically $E^\Hspace_\Yspace$-ergodic, whenever $H$ is both a non-Archimedean and TSI Polish group. As a consequence of Theorem \ref{Th:main}, we now have that $E^{[\mathbb{Z}]}_0$ is generically $E^\Hspace_\Yspace$-ergodic for every TSI Polish group $\Hspace$ and therefore $E^{[\mathbb{Z}]}_0$ is not classifiable by any TSI-group action. In fact, in Section \ref{S:Wreath} we define a {\bf$P$-jump operator} $E\mapsto E^{[P]}$, for every Polish permutation group $P$, which
is a common generalization of the Friedman–Stanley jump and the jump operators defined in \cite{ClemensCoskey}. It turns out that $P$-jumps are particularly natural in the context of the generalized Bernoulli shifts from \cite{KMPZ2019}: if $E$ is the orbit equivalence relation of the generalized Bernoulli shift of the Polish permutation group $Q$, then $E^{[P]}$ is  the orbit equivalence relation of the generalized Bernoulli shift of $(P \wreath Q)$. The anti-classification result for $E^{[\mathbb{Z}]}_0$  is a particular instance of the next corollary. Here,  $(P\wreath\Gspace)$ is just the Wreath product of $P$ and $\Gspace$; see Section \ref{S:Wreath}.

\begin{theorem}\label{Th:main2}
Let $\Xspace$ be a Polish $\Gspace$-space which has a dense orbit, and let $P$ be Polish group of permutations of a countable set $N$. If all $P$-orbits of $N$ are infinite, then the unbalanced graph of the Polish $(P\wreath\Gspace)$-space $\Xspace^{N}$ is generically semi-connected.
\end{theorem}

\begin{corollary}\label{Cor:1}
If $P,Q\leq S_{\infty}$ are both non-compact Polish permutation groups then the Bernoulli shift of $(P \wreath Q)$ has a generically unbalanced closed subshift.
\end{corollary}

Corollary \ref{Cor:1} provides many examples of orbit equivalence relations of CLI groups which are not classifiable by TSI-group actions. However, all such examples are classifiable by countable structures. The next corollary shows that the complexity class within the CLI region but outside of the TSI and the non-Archimedean region in Figure \ref{fig1} is  non-empty:

\begin{corollary}\label{Cor:2}
There is a Polish $\Gspace$-space $\Xspace$ of a CLI Polish group $\Gspace$ which is  turbulent and generically unbalanced.
\end{corollary}

We finally illustrate how our results apply to natural classification problems from topology and operator algebras. For any locally compact metrizable space $T$, consider the problem $(\mathrm{CTr}^*(T),\equiv_{\mathcal{M}})$, of classifying all {\bf separable continuous-trace $C^*$-algebras with spectrum $T$} up to {\bf Morita equivalence}; and the problem $(\mathrm{Bun}_{\mathbb{C}}(T),\simeq_{\mathrm{iso}})$, of classifying all {\bf Hermitian line bundles over $T$} up to {\bf isomorphism}. In Section \ref{S:Applications} we show that both problems are in general not classifiable by actions of TSI-groups, even when $T$ is a CW-complex. In contrast, recall that the \emph{base-preserving} versions $\equiv_{\mathcal{M}}^T$, and $\simeq_{\mathrm{iso}}^T$, of the above problems are always classifiable by TSI---in fact by abelian---group actions; \cite{BLP2019}.

\subsection{Structure of the paper} Theorem  \ref{Th:main}, which is the main result of this paper, is proved in Section \ref{S:TSI}.  The necessary background is developed independently in  Section \ref{S:homomo} and Section \ref{S:ergodicity}.  In Section \ref{S:Wreath} we assume Theorem \ref{Th:main} and provide ``in vitro" applications, such as Theorem \ref{Th:main2}, Corollary \ref{Cor:1}, and Corollary \ref{Cor:2}. In Section \ref{S:Applications} we discuss applications in topology and operator algebras. Finally, in Section \ref{S:final} we informally announce some extensions of this work beyond the TSI dividing line which is currently work in progress.

\section{Wreath products, $P$-jumps, and Bernoulli shifts}\label{S:Wreath}

In this section, for every Polish permutation group $P$ of a countable set we introduce a $P$-jump operator for $\Gspace$-spaces and equivalence relations, in the spirit of \cite{ClemensCoskey}. We then draw some connections with the theory of generalized Bernoulli shifts from \cite{KMPZ2019}. We also prove Theorem \ref{Th:main2}, and derive  Corollary \ref{Cor:1} and Corollary \ref{Cor:2}.

Let $N$ be a countable set. We denote by  $S(N)$ the group of all permutations of $N$. This is a Polish group, when endowed with the pointwise convergence topology. By a {\bf Polish group of permutations of $N$} we mean any  closed subgroup $P$ of $S(N)$. Since we are considering left actions, the pertinent action $P\times N\to N$ is given by
\[(p,n)\mapsto  p(n), \text{ for any map } p\colon N\to N \text{ in } P.\]
Let $P$ be as above and let  $\Gspace$ be an arbitrary Polish group. The group 
\[\Gspace^{N}:=\{ \boldsymbol{g}\colon N\to \Gspace\}, \text{ where }  (\boldsymbol{g}_2\boldsymbol{g}_1)(n):=\boldsymbol{g}_2(n)\boldsymbol{g}_1(n),\]
is Polish and there is a natural $P$-action $\varphi\colon P\times \Gspace^{N} \to \Gspace^{N}$ on $\Gspace^{N}$ by automorphisms:
\[\varphi(p,\boldsymbol{g})=\boldsymbol{g}^p, \text{ where } \boldsymbol{g}^p(n):= \boldsymbol{g}(p^{-1} ( n)).\]
The  {\bf Wreath product} $(P\wreath_{\mkern-6mu N} \, \Gspace)$ of $P$ and $\Gspace$, or simply $(P\wreath \Gspace)$,  is the group
\[P\wreath \Gspace :=P  \rtimes_{\varphi} \Gspace^{N}.   \]
Concretely, elements of $P\wreath \Gspace$ are all pairs $(p,\boldsymbol{g})$, where $p\in P$, $\boldsymbol{g}\in\Gspace^N$, and
\[(p_2,\boldsymbol{g}_2)\cdot (p_1,\boldsymbol{g}_1) := (p_2 p_1,  \varphi(p_1,\boldsymbol{g}_2)  \boldsymbol{g}_1), \text{ where } \big(\varphi(p_1,\boldsymbol{g}_2)  \boldsymbol{g}_1\big)(n)=\boldsymbol{g}_2(p_1^{-1}( n))\boldsymbol{g}_1(n).\]
If $\Xspace$ is a Polish $\Gspace$-space, then the {\bf $P$-jump} of $\Gspace\curvearrowright \Xspace$ is the $(P\wreath\Gspace)$-space $\Xspace^{N}$:
\[(p, \boldsymbol{g})\cdot \boldsymbol{x} := \boldsymbol{x}^{\boldsymbol{g},p}, \text{ where } \boldsymbol{x}^{\boldsymbol{g},p}(n):= \boldsymbol{g}^p(n)\cdot\boldsymbol{x}^p(n)= \boldsymbol{g}(p^{-1}(n))\cdot\boldsymbol{x}(p^{-1}(n)) .\]
Similarly, let $(\Xspace^N,E^{[P]})$ be the {\bf $P$-jump} of a classification problem $(\Xspace,E)$, where: 
\[\boldsymbol{x}E^{[P]}\boldsymbol{x}' \iff (\exists p\in P) \; (\forall n\in N) \;\; \boldsymbol{x}(p^{-1}(n)) E \boldsymbol{x}'(n)\]
Notice that if $P=S(N)$, then this is simply the Friedman–Stanley jump $E\mapsto E^{+}$, and if $P$ is the left regular representation of a countable  group $\Gamma$ as a subgroup of $S(\Gamma)$,  then this is the $\Gamma$-jump $E\mapsto E^{[\Gamma]}$, introduced in \cite{ClemensCoskey}. Clearly, the orbit equivalence relation of the $P$-jump of a $\Gspace$-space $\Xspace$ is the $P$-jump of the orbit equivalence relation of the same space. We may now proceed to the proof of Theorem \ref{Th:main2}.

\begin{proof}[Proof of Theorem \ref{Th:main2}.]
Let $\Xspace$ be a Polish $\Gspace$-space which has a dense orbit and let $P\leq S(N)$ be a Polish permutation group on a countable set $N$ with infinite orbits.

Fix any comeager set $C\subseteq \Xspace^N$. For any fixed $n\in N$ consider ``column" space  $\Xspace^{\{n\}}$ of all maps from the singleton $\{n\}$ to $\Xspace$. This is naturally isomorphic to the Polish $\Gspace$-space $\Xspace$. By intersecting $C$ with the appropriate comeager set we may assume without loss of generality that for all $\boldsymbol{x}\in C$ we have that:
\begin{equation}\label{eq:1}
\text{for all } n\in N, 
\text{ the orbit of } 
\boldsymbol{x}(n) \text{ is dense in } \Xspace^{\{n\}}.
\end{equation}
Here we use that $\{x\in\Xspace \mid [x] \text{ is dense} \}$ is a $G_{\delta}$ subset of $\Xspace$, and since by assumption it contains a dense orbit, it is comeager.

Let $E$ be the collection of all bijective maps from $N$ to the space $N\times\{0,1\}$, of two disjoint copies of $N$. Then $E$ is a Polish space with the pointwise convergence topology. Since every $P$-orbit of $N$ is infinite, by the Neumann's lemma \cite[Lemma 2.3]{Neumann1976} we have that for every finite $A,B \subseteq N$, there is $p\in P$ so that $(p\cdot A) \cap B =\emptyset$. As a consequence, for the generic $e\in E$ and for every $i\in\{0,1\}$: 
\begin{equation}\label{eq:2}
    \text{if } A\subseteq N \text{ is finite, then there is } p\in P \text{ so that }  e(p^{-1}(n))=(p^{-1}(n),i), \text{ for all } n\in A.
\end{equation}
Fix some $e\in E$ satisfying the above and consider the  map $\varphi\colon \Xspace^N \times \Xspace^N \to \Xspace^N$,
where $\varphi(\boldsymbol{x}_0,\boldsymbol{x}_1)$ is the function $\boldsymbol{x}\colon N\to \Xspace$, with $\boldsymbol{x}(n)=\boldsymbol{x}_0(m)$, if $e(n)=(m,0)$; and
$\boldsymbol{x}(n)=\boldsymbol{x}_1(m)$, if $e(n)=(m,1)$.

\begin{claim}
The map $\varphi\colon \Xspace^N \times \Xspace^N \to \Xspace^N$ is a homeomorphism of topological spaces.
\end{claim}
\begin{proof}[Proof of Claim.]
Since $e$ is bijective, every $\boldsymbol{x}\in \Xspace^N$ pulls back to some $(\boldsymbol{x}_0,\boldsymbol{x}_1)$ via $\varphi$ and every $(\boldsymbol{x}_0,\boldsymbol{x}_1)\in \Xspace \times \Xspace$ pushes forward to some $\boldsymbol{x}\in\Xspace$, i.e., $\varphi$ is a bijection. Continuity of  $\varphi$ and $\varphi^{-1}$ is straightforward.
\end{proof}

Consider the comeager subset $C^{\varphi}:=\varphi^{-1}(C) \bigcap (C\times C)$ of $\Xspace^N\times \Xspace^N$. By Kuratowski-Ulam, there is some $\boldsymbol{z}\in C$ so that $D:=\{\boldsymbol{y}\in \Xspace \mid (\boldsymbol{z},\boldsymbol{y})\in C^{\varphi}\}$ is a comeager subet of $\Xspace$. We are now done with the proof, since by the next claim, for every $\boldsymbol{x},\boldsymbol{y}\in D$ we have the following
path between $\boldsymbol{x}$ and $\boldsymbol{y}$ in $C$: 
\[\boldsymbol{x} \; \leftrightsquigarrow \;  \varphi(\boldsymbol{x},\boldsymbol{z}) \; \leftrightsquigarrow \; \boldsymbol{z}  \; \leftrightsquigarrow \;  \varphi(\boldsymbol{z},\boldsymbol{y})  \; \leftrightsquigarrow \; \boldsymbol{y}.\]

\begin{claim}
For all $\boldsymbol{x}_0,\boldsymbol{x}_1\in C$ we have that $\varphi(\boldsymbol{x}_0,\boldsymbol{x}_1)\leftrightsquigarrow \boldsymbol{x}_0$ and $\varphi(\boldsymbol{x}_0,\boldsymbol{x}_1)\leftrightsquigarrow \boldsymbol{x}_1$.
\end{claim}
\begin{proof}[Proof of Claim.]
It suffices to prove that $\varphi(\boldsymbol{x}_0,\boldsymbol{x}_1)\leftrightsquigarrow \boldsymbol{x}_0$ since the other case is symmetric. Set $\boldsymbol{x}=\boldsymbol{x}_0$ and $\boldsymbol{y}=\varphi(\boldsymbol{x}_0,\boldsymbol{x}_1)$. Let $V\subseteq(P\wreath \Gspace)$ be an open neighborhood of the identity and let $U\subseteq\Xspace$ be any non-empty open set. By shrinking both sets if necessary, we assume that there is a finite set  $A\subseteq N$, a map  $u\colon A \to \Xspace$, an open neighborhood $W\subseteq\Gspace$ of the identity of $\Gspace$, and some $\varepsilon>0$, so that 
\[V=\{(p,\boldsymbol{g})\in (P\wreath \Gspace)\mid \boldsymbol{g}(n)\in W, \text{ for all } n\in A, \text{ and } p \text{ fixes every } a \in A\}, \text{ and }\]
\[U=\{\boldsymbol{x}\in \Xspace^N \mid d\big(\boldsymbol{x}(n),u(n)\big)<\varepsilon, \text{ for all } n\in A\},\]
where $d$ is any metric on $\Xspace$ that is compatible with the topology.

By \eqref{eq:2} there is $p\in P$ so that $\big((p,\boldsymbol{1_\Gspace})\cdot\boldsymbol{x}\big)(n)=\big((p,\boldsymbol{1_\Gspace})\cdot\boldsymbol{y}\big)(n)$, for all $n\in A$. By \eqref{eq:1} we may set $g^x=g^y:=(p,\boldsymbol{g})$ for some $\boldsymbol{g}\in\Gspace^N$  so that in addition to
\begin{equation}\label{eq:3}
\big((p,\boldsymbol{g})\cdot\boldsymbol{x}\big)(n)=\big((p,\boldsymbol{g})\cdot\boldsymbol{y}\big)(n), \text{ for all } n\in A,    
\end{equation}
we also have $g^x \cdot\boldsymbol{x}, g^y\cdot \boldsymbol{y}\in U$.  But $V$ contains the following subgroup of $(P\wreath{}\Gspace)$: 
\[\{(1_P,\boldsymbol{h})\mid \boldsymbol{h}(n)=1_{\Gspace}, \text{ for all } n\in A\}.\]
By \eqref{eq:1} and \eqref{eq:3}, it now follows that $(g^y \boldsymbol{y})\in \overline{V \cdot(g^x \cdot \boldsymbol{x})}$ and  $(g^x \boldsymbol{x})\in\overline{V \cdot(g^y \cdot \boldsymbol{y})}$.
\end{proof}
\end{proof}

Let $P$ be a Polish permutation group of a countable set $N$. The {\bf generalized Bernoulli shift} of $P$ is the Polish $P$-space $\mathbb{R}^N$ where $(p,x)\mapsto  x^p$ with $x^p(n)=x(p^{-1}(n))$, for every $n\in N$. This indeed generalizes the classical Bernoulli shift $\Gamma\curvearrowright\mathbb{R}^{\Gamma}$ for countable disrete groups, where $(g, x)\mapsto g\cdot x$ with $(g\cdot x)(\gamma)=x(g^{-1}\gamma)$. In \cite{KMPZ2019} it was shown that the Borel reduction complexity of the orbit equivalence relation of $P\curvearrowright \mathbb{R}^N$ is often a reflection of the dynamical properties of $P$. 

In addition to $P$ above, consider a  Polish permutation group $Q$ on some set $M$. Notice that $P\wreath_{\mkern-6mu N}Q$ may be realized as a Polish permutation group on the set $N\times M$ via the faithful action $P\wreath_{\mkern-6mu N}Q\curvearrowright N\times M$ with
\[(p,\boldsymbol{q})\cdot(n,m)\mapsto (p(n),(\boldsymbol{q}(n))(m)).\]
Moreover, its associated generalized Bernoulli shift  $P\wreath_{\mkern-6mu N}Q\curvearrowright \mathbb{R}^{N\times M}$ is naturally isomorphic to the $P$-jump $P\wreath_{\mkern-6mu N}Q\curvearrowright (\mathbb{R}^{M})^N$ of the generalized Bernoulli shift $Q\curvearrowright \mathbb{R}^M$ of $Q$. In particular, since $E_0$ is Borel isomorphic to the orbit equivalence of the (classical) Bernoulli shift of $\mathbb{Z}$, we see that the $\mathbb{Z}$-jump  $E^{[\mathbb{Z}]}_0$ of $E_0$ can be identified with the orbit equivalence relation of the generalized Bernoulli shift \[\mathbb{Z}\wreath_{\mkern-6mu  \mathbb{Z}}\mathbb{Z}\curvearrowright \mathbb{R}^{\mathbb{Z}\times\mathbb{Z}}.\]
In \cite{ClemensCoskey} it was shown that $E_0^{[\mathbb{Z}]}$ is generically ergodic with respect to $E_0^{\omega}$. This was later generalized in \cite{Allison2020}, where it was shown that $E_0^{[\mathbb{Z}]}$ is generically ergodic with respect to any orbit equivalence relation of any non-Archimedean TSI Polish group action. The following is a special case of Corollary \ref{Cor:1} (in the form of Lemma \ref{Cor:1LemmaVersion} below). Notice that since $E_0^{[\mathbb{Z}]}$ has meager equivalence classes this implies that it is not classifiable by TSI-group actions.

\begin{corollary}
$E_0^{[\mathbb{Z}]}$ is generically ergodic with respect actions of TSI-groups.
\end{corollary}

We proceed  to the proof of Corollary \ref{Cor:1} after restating it in a more precise way:

\begin{lemma}\label{Cor:1LemmaVersion}
Let $P\leq S(N)$ and $Q\leq S(M)$ be  Polish permutation groups on countable sets $N,M$. If $P,Q$ are non-compact then the closed invariant subspace:
\[\{x\in \mathbb{R}^{M\times N}\mid x(m,n)\neq 0 \implies \text{ both orbits } [n]_P\subseteq N \text{ and } [m]_Q\subseteq M \text{ are infinite} \}\]
of the Bernoulli shift of $(P \wreath_{\mkern-6mu N} Q)$, has meager orbits and is generically ergodic with respect to actions of TSI Polish groups. 
\end{lemma}
\begin{proof}[Proof of Corollary \ref{Cor:1} in the form of Lemma \ref{Cor:1LemmaVersion}]
It is easy to see that the orbits are meager given that the $\mathrm{range}(x)$ is a countable subset of $\mathbb{R}$, for all $x\in \mathbb{R}^{N\times M}$. 

Let $M^{\infty}\subseteq M$ and $N^{\infty}\subseteq N$ be the collection of all points whose $Q$-orbit and $P$-orbit, respectively, is infinite. Let also $Q^*\leq S(M^{\infty})$  and $P^*\leq S(N^{\infty})$ be the image of $Q$ and $P$ respectively in their new representation. It is clear that the closed invariant subspace in the statement of the corollary is isomorphic to action: 
\[(P^*\wreath_{\mkern-6mu \mathbb{N}}Q^*)\curvearrowright \mathbb{R}^{N^{\infty}\times M^{\infty}},\]   
and as in the previous paragraph this is just the $P^{*}$-jump of the Benroulli shift of $Q^*$. But the Bernoulli shift of $Q^*$ is generically ergodic, since $Q^*$ is non-compact (see \cite{KMPZ2019}), and all orbits of $P^*$ are infinite. The rest follows from Theorem \ref{Th:main2}  
\end{proof}

Corollary \ref{Cor:1} gives many examples of classification problems which are classifiable by countable structures but not by actions of TSI groups. We may similarly find orbit equivalence relations which are classifiable neither by countable structures nor by TSI group actions. In fact we may do so while acting with a CLI group.

\begin{proof}[Proof of Corollary \ref{Cor:2}]
Let $(l_2,+)$ be the group of all sequences $(a_n)_n$ of reals which are square-summable. 
The action of $l_2$ on $\mathbb{R}^{\mathbb{N}}$ with $(a_n)_n\cdot(x_n)_n:=(a_n+x_n)_n$ is turbulent, see \cite{Gao2008}. In particular, it has meager orbits all of which happen to be dense. Let $\Gspace\curvearrowright \Xspace$ be the $\mathbb{Z}$-jump of $l_2\curvearrowright \mathbb{R}^{\mathbb{N}}$. It is easy to check that this space is turbulent and still has meager orbits. The rest follows from Theorem \ref{Th:main2} and the fact that the wreath product of CLI groups is CLI.
\end{proof}

\section{Definable classification induces homomorphism between unbalanced graphs}\label{S:homomo}

The following Theorem is the main result of this section.

\begin{theorem}\label{push_forward_arrows}
Suppose $\Xspace$ is a Polish $\Gspace$-space and $\Yspace$ is a Polish $\Hspace$ for Polish groups $\Gspace$ and $\Hspace$. For any Baire-measurable homomorphism $f : E^\Gspace_\Xspace \rightarrow E^\Hspace_\Yspace$, there is a $\Gspace$-invariant comeager set $C \subseteq \Xspace$ such that for any $x_0, y_0 \in C$, if $x_0 \leftrightsquigarrow y_0$ then $f(x_0) \leftrightsquigarrow f(y_0)$.
\end{theorem}

For the proof of this Theorem we will rely on two lemmas. The
following ``orbit-continuity" lemma is essentially \cite[Lemma 3.17]{Hjorth2010} modified as in the beginning of the proof of \cite[Theorem 3.18]{Hjorth2010}. For a direct proof see \cite{LupiniPanagio2018}, noting that the the set $C$ produced in that proof happens to be invariant.

\begin{lemma}\label{orbit_continuity}
Suppose $\Xspace$ is a Polish $\Gspace$-space and $\Yspace$ is a Polish $\Hspace$-space for Polish groups $\Gspace$ and $\Hspace$. For any Baire-measurable homomorphism $f : E^\Gspace_\Xspace \rightarrow E^\Hspace_\Yspace$, there is an invariant comeager set $C \subseteq \Xspace$ such that:
\begin{enumerate}
\item $f$ restricted to $C$ is continuous; and
\item for any $x_0 \in C$ and any open neighborhood $\mcW$ of the identity of $\Hspace$, there is an open neighborhood $\mcU$ of $x_0$ and an open neighborhood $\mcV$ of the identity of $\Gspace$ such that for any $x \in \mcU \cap C$ and for a comeager set of $g \in \mcV$, we have $f(gx) \in \mcW f(x)$.
\end{enumerate}
\end{lemma}

The next lemma says that the witnesses $g^x$ and $g^y$ in the definition of $x \leftrightsquigarrow y$ can be taken to be locally generic.
\begin{lemma}
For any Polish $\Gspace$-space $\Xspace$ and any $x, y \in \Xspace$, if $x \leftrightsquigarrow y$, then for any open neighborhood $\mcV$ of the identity of $\Gspace$ and any nonempty open neighborhood $\mcU \subseteq \Xspace$ of $x$ or $y$, there is a nonempty open set of $g^x \in \Gspace$ and a nonempty open set of $g^y \in \Gspace$ such that $g^x x, g^y y \in \mcU$, $g^y y \in \overline{\mcV  (g^x x)}$, and $g^x x \in \overline{(\mcV  (g^y y)}$.
\end{lemma}

\begin{proof}
Let $\mcU \subseteq \Xspace$ be an open set intersecting the orbits of $x$ or $y$, and let $\mcV$ an open neighborhood of the identity of $\Gspace$. Choose another open neighborhood $\mcV_0$ of the identity such that $\mcV_0^3 \subseteq \mcV$. By the definition we can find some $h^x, h^y \in \Gspace$ such that $h^x, h^y \in \mcU$ and $h^x x \in \overline{\mcV_0  (h^y y)}$ and $h^y  y \in \overline{\mcV_0  (h^x x)}$. 

However, observe that for any $g^x \in \mcV_0 h^x$ and $g^y \in \mcV_0 h^y$, we have $g^x x \in \overline{\mcV  (g^y y)}$ and $g^y y \in \overline{\mcV  (g^x x)}$. Also, the set of $g^x \in \mcV_0 h^x$ such that $g^x x \in \mcU$ and the set of $g^y \in \mcV_0 h^y$ such that $g^y y \in \mcU$ are both open and nonempty. Thus the sets of $g^x$ and $g^y$ satisfying the conditions contains nonempty open sets and thus are nonmeager.
\end{proof}

We turn now to the proof of Theorem \ref{push_forward_arrows}.

\begin{proof}[Proof of Theorem \ref{push_forward_arrows}]
Fix an arbitrary open neighborhood $\mcW$ of the identity of $\Hspace$, and an open set $\mcU \subseteq Y$ intersecting the orbit of $f(y_0)$ (the case that $\mcU$ intersects the orbit of $f(x_0)$ is similar).
By the invariance of $\leftrightsquigarrow$, it would suffice to prove the claim for the Baire-measurable homomorphism $x \mapsto h \cdot f(x)$ for any $h \in H$. Thus without loss of generality we may assume that $\mcU$ is a neighborhood of $f(y_0)$.
By the orbit-continuity lemma, we can find open neighborhoods $\mcU'$ of $y_0$ and $\mcW$ of the identity of $\Gspace$ such that for every $x \in \mcU' \cap C$, there is a comeager set of $v \in \mcV$ such that $f(v  x) \in \mcW  f(x)$. By shrinking $\mcU'$, we may assume that $f[\mcU'] \subseteq \mcU$.

Since $x_0 \leftrightsquigarrow y_0$, by the previous lemma, we may find some group elements $g^x, g^y \in \Gspace$ such that $g^x x_0, g^y y_0 \in \mcU' \cap C$, $g^x  x_0 \in \overline{\mcV  (g^x x)}$ and $g^y y_0 \in \overline{\mcV  (g^y  y)}$. 
Since $f$ is a homomorphism, we may fix group elements $h^x, h^y \in \Hspace$ such that $h^x  f(x_0) = f(g^x  x_0)$ and $h^y  f(y_0) = f(g^y  y_0)$, which are both elements of $\mcU$.

To see that $h^x f(x_0) \in \overline{\mcW  (h^y f(y_0))}$, fix an open neighborhood $\mcU_0$ of $h^x f(x_0)$. Notice $g^x  x_0 \in f^{-1}[\mcU_0] \cap C$ and the set $v \in \mcV$ such that $v  (g^y y_0) \in f^{-1}[\mcU_0]$ is nonempty open, thus we can choose one such that $f(v  (g^y y_0)) \in \mcW  f(g^y y_0) = \mcW  (h^y f(y_0))$. Checking that $h^y f(y_0) \in \overline{\mcW  (h^x f(x_0))}$ is the same.
\end{proof}

\section{Strong ergodicity properties and dynamical back and forth}\label{S:ergodicity}

Let $\Xspace$ be a Polish $\Gspace$-space. In this section we define 
 binary relations  $\precsim^{\alpha}_{\Gspace}$ and $\sim^\alpha_{\Gspace}$ on $\Xspace$,
for every $\alpha<\omega_1$. Intuitively, two points $x,y\in \Xspace$ satisfy $x\precsim^{\alpha}_{\Gspace}y$,  iff Player II has a non-losing strategy of rank $\alpha$ in a dynamical analogue of the classical Ehrenfeucht–Fra\"iss\'e game, where Player I is the ``spoiler" and starts by partially specifying $x$. In Proposition \ref{sim_to_invariance}, which is the main result of this section, we derive some strong ergodicity properties under the assumption 
 $x\precsim^{\alpha}_{\Gspace}y$.   Notice that the ideas developed in this section are similar in spirit to the content of \cite[Section 6.4]{Hjorth2010}.

\begin{definition}
Let $\Xspace$ be a Polish $\Gspace$-space, let $V$ be an open neighborhood of the identity of $\Gspace$ and let $x, y \in \Xspace$. 
By recursion on $\alpha<\omega_1$, we simultaneously define relations $x \precsim^{\alpha}_{V} y$ and $x \sim^\alpha_V y$ as follows:
\begin{enumerate}
    \item Let $x \precsim_{\mcV}^0 y$ holds exactly when $x \in \overline{\mcV  y}$;
    \item If $\precsim_{\mcV}^\alpha$ is defined, then  $x \sim^\alpha_\mcV y$  holds exactly when $x \precsim_\mcV^\alpha y$ and $y \precsim_\mcV^\alpha x$;
    \item Assume for some ordinal $\alpha$ that $\sim^\beta_{\mcW}$ is defined for every ordinal $\beta < \alpha$ and every open neighborhood $\mcW$ of the identity of $\Gspace$.
    Then,  $x \precsim^{\alpha}_\mcV y$ holds exactly when for every open neighborhood $\mcW$ of the identity of $\Gspace$ there exists some $v \in \mcV$, such that for every $\beta < \alpha$ we have that  $v y \sim^\beta_\mcW x$.  
\end{enumerate}
\end{definition}

Note that the relations $\precsim^\alpha_{\mcV}$ are not necessarily symmetric or transitive. 
The relations $\sim^\alpha_{\mcV}$ are symmetric by definition, but they are also not necessarily transitive. 
It's also worth noting that by an easy argument $\sim^1_\mcV \subseteq \sim^0_\mcV$, and then it follows that $\sim^\alpha_\mcV \subseteq \sim^\beta_\mcV$ and $\precsim^\alpha_\mcV \subseteq \precsim^\beta_\mcV$ for every $\beta \le \alpha$. Similarly, whenever $W \subseteq V$ are basic open neighborhoods of the identity of $G$, it is easy to see that $\sim^\alpha_W \subseteq \sim^\alpha_V$.

Before we state the main result of this section, recall the notation associated with the Vaught transforms. Let $A\subseteq\Xspace$ be a Baire-measurable set and let $V\subseteq \Gspace$ be any open set. If $x\in \Xspace$, we write $x\in A^{* V}$ if the set $\{v\in V \mid v x \in A\}$ is comeager in $V$. We write $x\in A^{\Delta V}$ if  the set $\{v\in V \mid v x \in A\}$ is non-meager in $V$. For basic properties of the Vaught transforms one may consult \cite{Gao2008}.

\begin{proposition}\label{sim_to_invariance}
Let $\Gspace$ an arbitrary Polish group and let $\Xspace$ be a Polish $\Gspace$-space. If $x \precsim^{\alpha}_\Gspace y$ and $\alpha \ge 1$, then for every $\BPi^0_{\alpha}$-set $A \subseteq \Xspace$ we have that
    \[x \in A^{\Delta \Gspace} \Rightarrow y \in A^{\Delta \Gspace}.\]
\end{proposition}

We start by recording some useful basic properties of the relations $\precsim^{\alpha}_{V}$.

\begin{lemma}\label{sim_facts}
Let $\Xspace$ be  a Polish $\Gspace$-space and let $\mcV, \mcW$ be open neighborhoods of the identity of $\Gspace$. For every  $\alpha<\omega_1$ and every $x, y, z \in \Xspace$ we have that:
\begin{enumerate}
    \item\label{sim_shift} If $x \precsim^\alpha_\mcV y$ and $g \in \Gspace$, then $g  x \precsim^\alpha_{g \mcV g^{-1}} g  y$; and
    \item\label{sim_trans} If $x \precsim^\alpha_\mcV y$ and $y \precsim^\alpha_\mcW z$, then $x \precsim^\alpha_{\mcV \mcW} z$.
\end{enumerate}
\end{lemma}

\begin{proof}
For (\ref{sim_shift}),   if $x \precsim^0_\mcV  y$, then $x=\lim_n h_n  y$ for some $h_n\in V$. But then, by continuity of the action we have $g  x=\lim_n gh_n  y=\lim_n gh_ng^{-1} g  y$; that is, $g  x \precsim^0_{g \mcV g^{-1}} g  y$. 
Assume now that $x \precsim^\alpha_{\mcV} y$ and let $W\subseteq \Gspace$ be an open neighborhood of the identity of $\Gspace$. Since 
$x \precsim^\alpha_{\mcV} y$ and $g^{-1} \mcW g$ is an open neighborhood of the identity of $\Gspace$, there is some $v \in \mcV$ such that for every $\beta < \alpha$, $v y \sim^\beta_{g^{-1} \mcW g} x$. 
By the inductive assumption, for every $\beta < \alpha$ we have that  $g v  y \sim^\beta_{\mcW} g x$. Hence $(g v g^{-1}) g  y \sim^\beta_{\mcW} g  x$ for every $\beta < \alpha$, with $gvg^{-1} \in g\mcV g^{-1}$, as desired.

For (\ref{sim_trans}), suppose first that $x \precsim^0_\mcV y$ and $y \precsim^0_\mcW z$. 
For an arbitrary open neighborhood $\mcU \ni x$, we can find an open neighborhood $\mcU' \ni y$ and some $v \in \mcV$ such that $v  \mcU' \subseteq \mcU$. 
Then we can find some $w \in \mcW$ such that $w  z \in \mcU'$, in which case $vw  z \in \mcU$, where $vw \in \mcV \mcW$. 
Thus $x \in \overline{\mcV \mcW  z}$.

Now suppose $x \precsim^\alpha_\mcV y$ and $y \precsim^\alpha_\mcW z$ for some $\alpha \ge 1$. Fix an open neighborhood $O$ of the identity of $\Gspace$, with the goal of showing for some $g \in \mcV \mcW$ that for every $\beta < \alpha$, $g z \sim^\beta_{O} x$.
Let $O_1$ be an open neighborhood of the identity of $\Gspace$ so that $O_1^2\subseteq O$. Since $x \precsim^\alpha_\mcV y$, there is $v \in \mcV$ such that for every $\beta<\alpha$, we have $v y \sim^\beta_{O_1} x$. Now find some $w \in \mcW$ such that for every $\beta < \alpha$, $w z \sim_{v^{-1} O_1 v}^\beta y$. By Lemma \ref{sim_facts}.(\ref{sim_shift}), for every $\beta < \alpha$ we get $vw  z \sim_{O_1}^\beta vy$.
Thus for every $\beta < \alpha$, we have $vwz \sim^\beta_{O^2_1} x$ by the induction hypothesis, and therefore $vw  z \sim_{O}^\beta x$ as desired.
\end{proof}

We may now proceed to the proof of Proposition \ref{sim_to_invariance}.

\begin{proof}[Proof of Proposition \ref{sim_to_invariance}]
Notice that  if $x \in A^{\Delta \Gspace}$, one can find an open neighborhood $\mcV$ of the identity of $\Gspace$ and a group element $g \in \Gspace$ such that $g  x \in A^{* \mcV}$. By Lemma \ref{sim_facts}.(\ref{sim_shift}), we have $gx \precsim_{\Gspace}^\alpha gy$. Then there is some $h \in \Gspace$ such that for every $\beta < \alpha$, $hg y \precsim_{\mcW}^\beta gx$, where $\mcW$ is chosen to be some open neighborhood of the identity of $\Gspace$ such that $\mcW^2 \subseteq \mcV$. Thus it suffices to prove the following claim, which tells us that $hg y \in A^{* \mcW}$ and thus $y \in A^{\Delta \Gspace}$.

\begin{claim}
Let $\mcV,\mcW$ be open neighborhoods of the identity of $\Gspace$ so that $\mcW^2 \subseteq \mcV$. If for some $\alpha \ge 1$ we have $A\in\BPi^0_{\alpha}(\Xspace)$ and $x \precsim^\beta_\mcW y$ for every $\beta < \alpha$, then $y \in A^{* \mcV}$ implies $x \in A^{* \mcW}$.
\end{claim}
\begin{proof}[Proof of Claim.]
We proceed by induction on $\alpha \ge 1$. First, suppose $x \precsim^0_\mcW y$ and that $y \in A^{* \mcV}$ for some closed set $A \subseteq \Xspace$. Assuming for the sake of contradiction that $x \not\in A^{* \mcW}$, one can pick some $w_0 \in \mcW$ such that $w_0  x \in \mcU$, where  $\mcU:=A^c$. By Lemma \ref{sim_facts}.(\ref{sim_shift}), we get $w_0  x \precsim^0_{w_0 \mcW w_0^{-1}} w_0  y$, in which case one can pick some $w_1 \in \mcW$ such that $w_0w_1  y \in \mcU$. So there is an open neighborhood of $w_0w_1$ of elements $w$ such that $w y \in \mcU$. Since $\mcW^2 \subseteq \mcV$, this contradicts $y \in A^{* \mcV}$.

Suppose now that, for some ordinal $\alpha > 1$ the claim is true below $\alpha$, that $x \precsim^{\beta}_\mcW y$ for every $\beta < \alpha$, and that $y \in A^{*\mcV}$ for some $\BPi^0_{\alpha}$ set $A$. Write $A = \bigcap_{n \in \omega} B_n$ where each $B_n \subseteq \Xspace$ is $\Sigma^0_{\beta_n}$ for some $\beta_n < \alpha$. Assume for the sake of contradiction that $x \not\in A^{*\mcW}$. 
Then there is some $w_0 \in \mcW$ and an open neighborhood $\mcW_0$ of the identity of $\Gspace$, as well as some $n_0 \in \omega$, such that $\mcW_0w_0 \subseteq \mcW$ and $w_0  x \in (\Xspace \setminus B_{n_0})^{* \mcW_0}$.
Choose an open neighborhood $\mcW_1$ of the identity of $\Gspace$ such that $\mcW_1^2 \subseteq \mcW_0$.
Then we can find some $w \in \mcW$ such that for every $\beta < \alpha$, $wy \precsim_{w_0^{-1} \mcW_1 w_0}^\beta x$. 
By Lemma \ref{sim_facts}.(\ref{sim_shift}) we have $w_0w  y \precsim^{\beta}_{\mcW_1} w_0 x$.
Thus by the induction hypothesis applied to $w_0 x \in (\Xspace \setminus B_{n_0})^{* \mcW_0}$, we have $w_0w  y \in (\Xspace \setminus B_{n_0})^{*\mcW_1}$. But $\mcW_1 w_0w \cap \mcV \neq \emptyset$, contradicting that $y \in A^{* \mcV}$.
\end{proof}
\end{proof}

\section{Dynamics of TSI Polish groups}\label{S:TSI}

In this section, we derive some consequences for the relations $\leftrightsquigarrow$ and $\precsim^{\alpha}_{\Hspace}$, when these relations have been defined on a Polish $\Hspace$-space $\Yspace$, where $\Hspace$ is a TSI Polish group. We then conclude with the proof of Theorem \ref{Th:main}.

Throughout this section $\Hspace$ is a TSI Polish group and $\Yspace$ is a Polish $\Hspace$-space.  We  also fix a countable basis $\mcB$ of open, symmetric, and conjugation-invariant neighborhoods of the identity of $\Hspace$. We  assume that $\Gspace\in \mcB$ and that for any $\mcV \in \mcB$, we have $\mcV^2 \in \mcB$.

\begin{lemma}\label{L:TSI_facts}
Let $\Hspace$ be a Polish TSI group and $\Yspace$ a Polish $\Hspace$-space. We have:
\begin{enumerate}
    \item \label{ref_1} \label{arrow_to_sim} if $x \leftrightsquigarrow y$, then $x \sim^1_\Hspace y$;
    \item \label{ref_2} $\precsim^\alpha_{V}$ and $\sim^\alpha_{V}$ coincide for all $V\in\mathcal{B}$ and all ordinals $\alpha>0$; and
    \item \label{ref_3} $\sim^\alpha_{\Hspace}$ is an equivalence relation for all ordinals $\alpha>0$.
\end{enumerate}{}
\end{lemma}
\begin{proof}
For (\ref{ref_1}), by taking $U$ to be $X$ in the definition of $x \leftrightsquigarrow y$, we can find $g^x, g^y$ such that $g^y y \sim^0_V g^x x$. By Lemma \ref{sim_facts}(\ref{sim_shift}) and the conjugation-invariance of $V$, we get $g y \sim^0_V x$ for $g = (g^x)^{-1}g^y$ as desired.

For (\ref{ref_2}), it follows immediately from the definitions that $\sim_V^\alpha \subseteq \precsim_V^\alpha$, so it suffices to show that $\precsim_V^\alpha \subseteq \sim_V^\alpha$. This amounts to showing that $\precsim_V^\alpha$ is symmetric.
If $x \precsim^\alpha_{V} y$ and $W\in\mathcal{B}$, then there is some $v \in V$ such that for every $\beta < \alpha$, we have that $ vy \sim^\beta_{W} x$. By Lemma \ref{sim_facts}(\ref{sim_shift}) it follows by the conjugation-invariance of $\mcW$ that $v^{-1} x \sim^\beta_{W}y$ for all $\beta<\alpha$. Since $V$ is symmetric, $v^{-1} \in V$. Hence, $ y \precsim^\alpha_{V} x$.

For (\ref{ref_3}), transitivity follows from  Lemma \ref{sim_facts}(\ref{sim_trans}), and symmetry by (\ref{ref_2}) above.

\end{proof}

In the rest of this section we will want to refer to the relations $\precsim^\alpha_\mcV$ and $\sim^\alpha_\mcV$ computed according to multiple Polish topologies on the same space. For any topology $\sigma$ making $(\Yspace, \sigma)$ a Polish $\Hspace$-space, we will use the notation $\precsim^{\alpha, \sigma}_\mcV$ and $\sim^{\alpha, \sigma}_\mcV$ to refer to the relations $\precsim^\alpha_\mcV$ and $\sim^\alpha_\mcV$ as computed in that space. We will use $\tau$ 
to refer to the original topology on $\Yspace$, but keep denoting $\precsim^{\alpha, \tau}_\mcV$ and $\sim^{\alpha, \tau}_\mcV$ simply by $\precsim^{\alpha}_\mcV$ and $\sim^\alpha_\mcV$. For every $V\in \mathcal{B}$, $c\in\Yspace$, and ordinal $\beta > 0$, let
\[A^{\beta}_V(c):=\{d\in \Yspace\mid \forall \gamma < \beta, d \sim^{\gamma}_\mcV c\}.\]
The following technical lemma will be useful.

\begin{lemma}\label{L:last}
Let $\sigma$ be an additional topology on $\Yspace$ so that both $\Yspace$ and $(\Yspace, \sigma)$ are Polish $\Hspace$-spaces. Let $a, b, c \in \Yspace$ and let $\alpha \ge 2$ be an ordinal so that: 
\begin{enumerate}
    \item for every $\beta < \alpha$ we have  $a \sim_\Hspace^{\beta} c$ and $b \sim_\Hspace^{\beta} c$; and
    \item for every $\mcV, \mcW \in \mcB$, the set $(A^\alpha_\mcV(c))^{\Delta \mcW}$ is  in $\sigma$.
\end{enumerate}
Then
\[ a \sim_\Hspace^{1,\sigma} b \implies a \sim_\Hspace^{\alpha} b.\]
\end{lemma}
\begin{proof}
Let $\mcV$ be an arbitrary open neighborhood of the identity of $\Hspace$. Our goal is to show that there is some $h \in \Hspace$ such that for every $\beta < \alpha$, $h b \sim_\mcV^{\beta} a$. To that end, let $\mcV_0\in\mathcal{B}$ with $\mcV_0^{-1} \subseteq \mcV$, and fix $h \in \Hspace$ such that $h b \sim^{0, \sigma}_{\mcV_0} a$. Fix any ordinal $\beta$ such that $\beta< \alpha$. We claim that $h b \sim^{\beta}_\mcV a$. To see this, let  $\mcW\in\mathcal{B}$. We will find some $v \in \mcV$ such that $v a \sim^{ \gamma}_\mcW h b$, for all $\gamma<\beta$.

Let $\mcW_0\in\mathcal{B}$ so that $\mcW_0^4 \subseteq \mcW$. Because $a \sim^\beta_H c$ we may choose some $g \in \Hspace$ such that for every $\gamma < \beta$, $g  a \sim_{\mcW_0}^{\gamma} c$, in which case $g a \in A^\beta_{\mcW_0}(c)$.
By Lemma \ref{sim_facts}(\ref{sim_trans}), we can see that $wga \in A^\beta_{\mcW_0^2}(c)$ for any $w \in W_0$.
Thus $a \in (A^\beta_{\mcW_0^2}(c))^{\Delta \mcW_0 g}$. 
Since $(A^\beta_{\mcW_0^2}(c))^{\Delta \mcW_0 g}$ is an open neighborhood of $a$ in $\sigma$, there is $v \in \mcV_0$ such that $vh  b \in (A^\beta_{\mcW_0^2}(c))^{\Delta \mcW_0 g}$. 
By definition, for some $w \in \mcW_0$, we have $\forall \gamma < \beta, wgvh  b \sim_{\mcW_0^2}^{\gamma} c$. 
For every $\gamma < \beta$, since $ga \sim^\gamma_{W_0} c$ and $wgvh  b \sim_{\mcW_0^2}^{\gamma} c$,  by Lemma \ref{sim_facts}(\ref{sim_trans}), we have that $g a \sim^{ \gamma}_{\mcW_0^3} wgvh  b$, and thus $g a \sim^{\gamma}_{\mcW_0^4} gvh  b$.
By Lemma \ref{sim_facts}(\ref{sim_shift}), we get $v^{-1} a \sim^{\gamma}_{\mcW_0^4} h  b$ and thus $v^{-1} a \sim^{\gamma}_{\mcW} h  b$ for every $\gamma < \beta$ as desired.
\end{proof}

We may now proceed to the proof of Theorem \ref{Th:main}.

\begin{proof}[Proof Theorem \ref{Th:main}]
Let $f : \Xspace \rightarrow \Yspace$ be a Baire-measurable homomorphism from $E^\Gspace_\Xspace$ to $E^\Hspace_\Yspace$, for some Polish $\Hspace$-space $\Yspace$, where $\Hspace$ is a TSI Polish group. 

\begin{claim} For all $1 \le \alpha<\omega_1$ there is  comeager $C_{\alpha}\subseteq\Xspace$ so that for all $x, y \in C_{\alpha}$,
\[f(x) \sim^{\alpha}_H f(y).\]
\end{claim}
\begin{proof}
For $\alpha=1$, by Lemma \ref{push_forward_arrows} we have a comeager set $D$ such that for any $x, y \in D$, if $x \leftrightsquigarrow y$, then $f(x) \leftrightsquigarrow f(y)$. Since $(\Xspace/\Gspace,\leftrightsquigarrow)$ is generically semi-connected, we can find a comeager set $C \subseteq D$ such that for any $x, y \in C$, there is a $\leftrightsquigarrow$-path between $x$ and $y$ through $D$. In particular, by Lemma \ref{L:TSI_facts}(\ref{ref_3}), for any $x, y \in C$, we have $f(x) \sim^1_\Hspace f(y)$. So we may set $C_1:=C$.

Assume now that for some countable $\alpha\geq 2$ we have that defined $C_{\beta}$ for all $\beta<\alpha$ as in the claim. 
Fix some $z \in \bigcap_{\beta<\alpha} C_{\beta}$.
Observe that the set $A^{\alpha}_\mcV(f(z)) = \{a \in \Yspace \mid \forall \beta < \alpha, a \sim^{\beta}_\mcV f(z) \}$ is Borel, and $f(x) \in A^{\alpha}_\mcV(f(z))$ for every $x \in \bigcap_{\beta < \alpha} C_\beta$.
Find a new topology $\sigma$ on $\Yspace$ such that: $(A^\alpha_\mcW)^{\Delta \mcV}$ is open for every $\mcV, \mcW \in \mcB$;  $(\Yspace, \sigma)$ is a Polish $\Hspace$-space; and $\sigma$ generates the same Borel sets as $\tau$ (see \cite[Lemma 4.4.3]{Gao2008}). 
Applying Lemma \ref{push_forward_arrows} and generic semi-connectedness of $(X/G, \leftrightsquigarrow)$ as in the previous paragraph with the space $(\Yspace, \sigma)$ in place of $\Yspace$, we can find a comeager set $C \subseteq \bigcap_{\alpha<\beta}C_{\alpha}$ so that for every $x, y \in C$, $f(x) \sim_\Hspace^{\sigma, 1} f(y)$.  
By Lemma \ref{L:last}, taking $c$ to be $f(z)$, we have that $f(x) \sim_\Hspace^{\alpha} f(y)$, for every $x, y \in C$. Set $C_{\alpha}:=C$.
\end{proof}

\begin{claim}
There is a comeager set $C \subseteq \Xspace$ and a countable ordinal $\lambda$ such that for every $x \in C$, $[f(x)]$ is $\BPi^0_\lambda$
\end{claim}

\begin{proof}
By \cite[Theorem 7.3.1]{BeckerKechris1996}, there is a Baire-measurable function $g : \Yspace \rightarrow 2^{\omega \times \omega}$ and a sequence $\{A_\zeta\}_{\zeta \in \omega_1}$ of pairwise-disjoint invariant Borel sets such that $E^\Hspace_\Yspace \upharpoonright A_\zeta$ is Borel for every $\zeta \in \omega_1$, and $g(y)$ is a well-order such that $y \in A_{|g(y)|}$ for every $y \in \Yspace$. As $g$ is Baire-measurable and thus so is $g \circ f$, we may find a dense $G_\delta$ subset $C \subseteq \Xspace$ such that $(g \circ f) \upharpoonright C$ is continuous. Applying $\BSigma^1_1$-boundedness (see \cite[Theorem 31.2]{Kechris1995}) to $(g \circ f) \upharpoonright C$, we can find a countable ordinal $\gamma$ such that $f[C] \subseteq A_\gamma$. Then since $E^\Hspace_\Yspace \upharpoonright A_\gamma$ is Borel, we can find the countable ordinal $\lambda$ with the property that $[f(x)]$ is $\BPi^0_\lambda$ for every $x \in C$.
\end{proof}

Fix now $\lambda$ and $C \subseteq \Xspace$ as in the claim such that for any $x \in C$, $[f(x)]$ is $\BPi^0_\lambda$. By the previous claim, the set $D:=C\cap C_{\lambda}$ is comeager and for any $x, y \in D$, $f(x) \sim^\lambda_\Hspace f(y)$. 
By Lemma \ref{sim_to_invariance}, this means that for any $\BPi^0_\lambda$ set $A$, $f(x) \in A^{\Delta \Hspace}$ iff $f(y) \in A^{\Delta \Hspace}$. Since $[f(x)]$ is $\BPi^0_\lambda$ for all $x\in D$, we have that every $x\in D$ maps to the same $\Hspace$-orbit in $\Yspace$.
\end{proof}

\section{Applications}\label{S:Applications}

In this section we illustrate how the ``in vitro" results we have  developed so far apply to natural classification problems from topology and operator algebras.  We start by reviewing some definitions regarding fibre bundles. We then show that \emph{coordinate free isomorphism between Hermitian line bundles} and \emph{Morita equivalence between continuous-trace $C^*$-algebras} are not classifiable by TSI-group actions.

\subsection{The Polish space of locally trivial fibre bundles}\label{SS:Bundles}

Let $B$ be a locally compact metrizable topological space, and let $F$ be a Polish $G$-space, for some Polish group $G$. A  locally trivial fibre bundle over $B$ with fibre $F$ and structure group $G$, or simply a {\bf fibre bundle over $B$} consists of  a Polish space $E$;  a continuous map $p\colon E\to B$; a locally finite open cover $\mathcal{U}$ of $B$;
and a homeomorphism $h_U\colon p^{-1}(U)\to U\times F$, for each $U\in\mathcal{U}$; so that: 
\begin{enumerate}
    \item if $b\in U\in \mathcal{U}$, then $h_U$ restricts to a homeomorphism from $p^{-1}(b)$ to $\{b\}\times F$; 
    \item if $U,V\in\mathcal{U}$, there is a  contiunous $t_{(U, V)}\colon U\cap V\to G$, so that for all $b\in U\cap V$,\[(h_V\circ h_U^{-1})(b,f)=(b,t_{(U,V)}(b)f)\]
\end{enumerate}
The maps $h_U$ above are called {\bf charts} and $t_{(U,V)}$ are called the {\bf transition maps}. Notice that we can always choose $\mathcal{U}$ to be a subset of some fixed countable basis $\mathcal{B}$ of the topology of $B$, and we can recover $E$ as the colimit of the above separable data (together with a $1$-cocycle condition). Hence,  we may form the Polish space  $\mathrm{Bun}(B,G,F)$,  of all locally trivial fibre bundle over $B$, with fibre $F$, and structure group $G$, as a $G_{\delta}$ subset of the Polish space
\[2^{\mathcal{B}}\times \prod_{(U,V)\in\mathcal{B}^2} C(U\cap V, G) \]

There are two natural classification problems on $\mathrm{Bun}(B,G,F)$: the {\bf isomorphism} relation $\simeq_{\mathrm{iso}}$; and the {\bf isomorphism over $B$} relation $\simeq_{\mathrm{iso}}^B$. First, notice that 
If $p,q\colon E\to B$ are elements of $\mathrm{Bun}(B,G,F)$, then we may always choose a common open cover $\mathcal{U}$ of $B$ so that $p$ and $q$ are locally trivialized by some $(h_U), (t_{(U,V)})$ and  $(k_U), (s_{(U, V)})$, respectively, with $U,V\in\mathcal{U}$. We write $p\simeq_{\mathrm{iso}}q$, if there are homeomorphisms $\pi\colon E\to E$, $\rho\colon B\to B$, and a continuous $e_{(U, V}\colon U \cap  \rho^{-1}(V) \to G$, so that $q \circ \pi = \rho \circ p$, and for all $U,V\in\mathcal{U}$, for all $b\in U\cap \rho^{-1}(V)$, and for all $f\in F$ we have that  \[  (h_V\circ \pi \circ h^{-1}_U)(b,f)= \big(\rho(b),e_{(U, V)}(b)f\big).\]
We write $p\simeq_{\mathrm{iso}}^{B}q$ if $\rho$ above can be taken to be $\mathrm{id}_B$.

\subsection{Isomorphism of Hermitian line bundles}\label{SS:Herm} Let $B$ be a locally compact metrizable space. By a {\bf Hermitian line bundle over $B$} we mean any locally trivial fibre bundle over $B$ with fibre $F:=\mathbb{C}$ and structure group $G:=\mathrm{U}(\mathbb{C})=\mathbb{T}$ being the unitary group of $\mathbb{C}$ acting on $\mathbb{C}$ with rotations.  Let $\mathrm{Bun}_{\mathbb{C}}(B)$ be the standard Borel space of all Hermitian line bundles over $B$. By a result of \cite{BLP2019}, $\simeq_{\mathrm{iso}}^{B}$ is  classifiable by TSI group actions:

\begin{proposition}[\cite{BLP2019}(Corollary 5.12.)]
The problem $(\mathrm{Bun}_{\mathbb{C}}(B),\simeq_{\mathrm{iso}}^{B})$ is classifiable by non-Archimedean, abelian group actions.
\end{proposition}

In contrast, for the relation  $\simeq_{\mathrm{iso}}$ we have the following result.

\begin{corollary}\label{Cor:HermiNotTSI}
There exists a locally compact metrizable topological space $B$, so that $(\mathrm{Bun}_{\mathbb{C}}(B),\simeq_{\mathrm{iso}})$ is not classifiable by TSI group actions. In fact, $B$ can be taken to be the geometric realization of a countable, locally-finite, CW-complex.
\end{corollary}
\begin{proof}
 The dyadic solenoid $\Sigma$ is the inverse limit of the inverse system $(\mathbb{T}_i, f^j_i)$ where  $\mathbb{T}_i:=\mathbb{T}$ is the unit circle, viewed as a multiplicative subgroup of $\mathbb{C}$, and $f^j_i\colon \mathbb{T}_j\to\mathbb{T}_i$ is the two-fold cover $z\mapsto z^2$. Let $C$ be the homotopy limit of the same inverse system. This is formed by taking the disjoint union of the spaces:
 \[\mathbb{T}_0\times [0,1], \mathbb{T}_1 \times [0,1], \mathbb{T}_2 \times [0,1], \ldots\]
and identifying the point $(z,0)\in \mathbb{T}_{i+1}\times[0,1]$   with  the point $(z^2,1)\in \mathbb{T}_i\times[0,1]$, for each $i\geq 0$. Clearly $C$ is a locally finite CW-complex. 

Recall now that the quotient $\mathrm{Bun}_{\mathbb{C}}(C)/\simeq_{\mathrm{iso}}^{C}$ is in bijective correspondence with the first \v{C}ech cohomology group $H^1(C,\mathbb{T})$ of $C$ with coefficients from $\mathbb{T}$; \cite[Proposition 4.53]{Morita}. Utilizing the short exact sequence $0\to \mathbb{Z}\to \mathbb{R}\to\mathbb{T}\to 0$ associated to the universal covering of $\mathbb{T}$ the later is isomorphic to the second \v{C}ech cohomology group $H^2(C,\mathbb{Z})$ of $C$ with coefficients from $\mathbb{Z}$ \cite[Theorem 4.42]{Morita}. By Steenrod duality \cite{Steenrod}, and since $C$ is homotopy equivalent to a solenoid complement $S^{3}\setminus \Sigma$, $H^2(C,\mathbb{Z})$ isomorphic to $0$-th Steenrod homology group $H_0(C,\mathbb{Z})$. 

In \cite{BLP2019} it was shown that the \v{C}ech cohomology groups for locally compact metrizable spaces, as well as the Steenrod homology groups for compact metrizable spaces, are quotients of Polish $G$-spaces. Moreover all the computations described in the previous paragraph lift to Borel reductions on the level of Polish spaces; see \cite[Lemma 2.14, Theorem 3.12, and Section 5.5]{BLP2019}.  By \cite[Proposition 4.2]{BLP2019}  we have $(\mathrm{Bun}_{\mathbb{C}}(C),\simeq_{\mathrm{iso}}^{C})$  is Borel bireducible with the orbit equivalence relation of the action of $\mathbb{Z}$ on its dyadic profinite completion $\mathbb{Z}_{\boldsymbol{2}}$ by left-translation, which is Borel bireducible to the equivalence relation $(2^{\mathbb{N}},E_0)$ of eventual equality of binary sequences. It turns out that $(\mathrm{Bun}_{\mathbb{C}}(C),\simeq_{\mathrm{iso}})$ is also Borel bireducible to  $(2^{\mathbb{N}},E_0)$. Indeed, by Borel functoriality of the definable \v{C}ech cohomology---this is proved in \cite{BLPBook}, but for CW-complexes it can be checked by hand---the action of $\mathrm{Homeo}(C)$ on $C$ induces definable endomorphisms of 
\[0\to \mathbb{Z}\to \mathbb{Z}_{\boldsymbol{2}}\to \mathbb{Z}_{\boldsymbol{2}}/\mathbb{Z}\to 0,\]
as in \cite[Section 5.3]{BLP2019}.  It follows by \cite[Proposition 5.6]{BLP2019} that  $(\mathrm{Bun}_{\mathbb{C}}(C),\simeq_{\mathrm{iso}})$ is also Borel bireducible to  $(2^{\mathbb{N}},E_0)$.

Fix some point $p$ in $C$, say the one corresponding to $(1,0)\in \mathbb{T}_0\times[0,1]$, and let $B$ be the CW-complex which is attained by taking the disjoint union of $\mathbb{Z}$-many copies $(C_k)$ of $C$ and connecting the point $p$ of $C_k$ to the point $p$ of $C_{k+1}$ by gluing on them the endpoints of a homeomorphic copy of the interval $[0,1]$. Hence, $B$ is a $\mathbb{Z}$-line of intervals.  Every homeomorphism of $C$ acts on the indexing copy of $\mathbb{Z}$ by the group $\Gamma:=(\mathbb{Z}/2\mathbb{Z}) \rtimes \mathbb{Z}$ in the obvious way. It is easy to see that the $\Gamma$-jump $E^{[\Gamma]}_0$ of $E_0$ reduces to $(\mathrm{Bun}_{\mathbb{C}}(B),\simeq_{\mathrm{iso}})$. By Theorem \ref{Th:main} and Theorem \ref{Th:main2}  we have that  $E^{\Gamma}_0$ is not classifiable by TSI-group actions.
\end{proof}{}

\subsection{Morita equivalence of continuous-trace $C^*$-algebras}\label{SS:C^*}

In what follows we will only consider separable $C^*$-algebras $A$ whose spectrum $\widehat{A}$ is Hausdorff. This implies that  $\widehat{A}$ is  a locally compact metrizable space. By the Gelfand-Naimark theorem, the subclass of all such commutative $C^*$-algebras is ``locally-concretely" classified via the assignment $A\mapsto \widehat{A}$: every two commutative $C^*$-algebras with homeomorphic spectrum are isomorphic. 
The unique up to isomorphism such commutative   $C^*$-algebra of spectrum $S$ is simply the algebra $C_0(S,\mathbb{C})$, of all continuous maps from $S$ to $\mathbb{C}$ which vanish at infinity. 
It turns out the Borel complexity of similar ``local" classification problems increases drastically even in the case of \emph{continuous-trace} $C^*$-algebras, which is the closest it gets to being commutative. For more on the general theory of $C^*$-algebras than we provide here, see \cite{Morita, Blackadar}.

Let $S$ be a locally compact metrizable space and let $\mathcal{K}(\mathcal{H})$ be the $C^*$-algebra of all compact operators on the separable Hilbert space. Two  $C^*$-algebra $A,B$ with spectrum $S$ are {\bf Morita equivalent} if $A\otimes \mathcal{K}(\mathcal{H})$ and $B\otimes \mathcal{K}(\mathcal{H})$ are isomorphic as $C^*$-algebras. In general, this isomorphism may only preserve the spectrum up to homeomorphism. When this induced homeomorphism can be taken to be $\mathrm{id}_S$ then we say that $A$ and $B$ are {\bf Morita equivalent over $S$}. Any $C^*$-algebra $A$ with Hausdorff spectrum can be endowed with $C_0(S)$-module structure, where $C_0(S)$ is the collection of all continuous $f\colon S\to\mathbb{C}$ which vanish at infinity. As a consequence, $A$ and $B$ are Morita equivalent over $S$ if and only if they are isomorphic via a $C_0(S)$-linear map.

Let $\mathrm{CTr}^*(S)$ be the space of all {\bf continuous-trace $C^*$-algebras with spectrum $S$}. These are all $C^*$-algebras $A$, for which there is an open cover $\mathcal{U}$ of $S$ consisting of relatively compact sets so that, for all $U\in\mathcal{U}$, if $A^{\overline{U}}$ is the quotient algebra induced by $\overline{U}\subseteq S$, then $A^{\overline{U}}$ is Morita equivalent to $C(\overline{U},\mathbb{C})$ over $\overline{U}$; \cite[Proposition 5.15]{Morita}. In other words
continuous-trace $C^*$-algebras are precisely the algebras which are locally Morita equivalent to commutative. 
It turns out that any algebra $A\in \mathrm{CTr}^*(S)$ can be identified with a locally trivial fibre bundle over $S$, whose fibre $F$ is the $C^*$-algebra $\mathcal{K}(H)$ of all compact operators of a separable (potentially finite dimensional) Hilbert space, and the strcture group $G$ is $\mathrm{Aut}(\mathcal{K}(H))$; \cite[IV.1.7.7, IV.1.7.8]{Blackadar}. Hence, similarly to subsection \ref{SS:Bundles}, we may view $\mathrm{CTr}^*(S)$ as a Polish space. The space $\mathrm{CTr}_{Stable}^*(S)$ of all {\bf stable continuous-trace $C^*$-algebras with spectrum $S$} is the subspace $\mathrm{Bun}(B,\mathrm{Aut}(\mathcal{K}(\mathcal{H})),\mathcal{K}(\mathcal{H}))$  of $\mathrm{CTr}^*(S)$---where $\mathcal{H}$ is  the separable infinite dimensional Hilbert space.

We consider the following two  classification problems: let $(\mathrm{CTr}^*(S),\equiv_{\mathcal{M}})$ be the problem of classifying all elements of $\mathrm{CTr}^*(S)$ up to Morita equivalence; and let $(\mathrm{CTr}^*(S),\equiv_{\mathcal{M}}^S)$ be the problem of classifying all elements of $\mathrm{CTr}^*(S)$ up to Morita equivalence over $S$. By a result of \cite{BLP2019}, $\equiv_{\mathcal{M}}^S$ is  classifiable by TSI group actions:

\begin{proposition}[\cite{BLP2019}(Corollary 5.14.)]
The problem $(\mathrm{CTr}^*(S),\equiv_{\mathcal{M}}^S)$ is classifiable by non-Archimedean, abelian group actions.
\end{proposition}

In contrast, for the relation  $\equiv_{\mathcal{M}}$ we have the following result.

\begin{corollary}
There exists a locally compact metrizable topological space $S$, so that $(\mathrm{CTr}^*(S),\equiv_{\mathcal{M}})$ is not classifiable by TSI group actions. In fact, $S$ can be taken to be the geometric realization of a countable, locally-finite, CW-complex.
\end{corollary}
\begin{proof}
First notice that the Borel map implementing $A\mapsto A\otimes \mathcal{K}(\mathcal{H})$, is a Borel reduction, inducing a bijection from $\mathrm{CTr}^*(S)/\equiv_{\mathcal{M}}$ to $\mathrm{CTr}_{Stable}^*(S)/\equiv_{\mathcal{M}}$. Hence, it sufffices to consider the problem  $(\mathrm{CTr}_{Stable}^*(S),\equiv_{\mathcal{M}})$ instead.

By the Dixmier-Douady classification theorem  we have that $\mathrm{CTr}_{Stable}^*(S)/\equiv_{\mathcal{M}}^S$ is in bijective correspondence with the third \v{C}ech cohomology group $H^3(S,\mathbb{Z})$ of $S$ with coefficients from $\mathbb{Z}$; see \cite[Theorem 5.29]{Morita}. It follows that $\mathrm{CTr}_{Stable}^*(S)/\equiv_{\mathcal{M}}$ is in bijective correspondence with $H^3(S,\mathbb{Z})/\Gamma$, where $\Gamma$ is the group of all automorphism of  $H^3(S,\mathbb{Z})$ induced by the action of $\mathrm{Homeo}(S)$ on $H^3(S,\mathbb{Z})$; see \cite[IV.1.7.15]{Blackadar}.  Similarly to the proof of Corollary \ref{Cor:HermiNotTSI}, the Dixmier-Douady correspondence and all the cohomological manipulations lift to Borel reductions on the level of a appropriate Polish spaces; see \cite{BLP2019}. The rest of the proof follows as in Corollary \ref{Cor:HermiNotTSI}: let  $D$ be the suspension $C\times [0,1]/\sim$ of the space $C$ which we defined in the proof of Corollary \ref{Cor:HermiNotTSI}, and let $S$ be attained from $D$ in the same way that $B$ was attained from $C$ in the same proof. Notice that by properties of the suspension we have that $H^3(D,\mathbb{Z})=H^2(C,\mathbb{Z})$.
\end{proof}

\begin{remark}
The complexity of the $(\mathrm{CTr}^*(S),\equiv_{\mathcal{M}})$ has been studied in \cite{BLPBook} for several spaces $S$. For example, it is shown that, when $S$ is the homotopy limit coming from the defining inverse system of a $d$-dimensional solenoid, then $(\mathrm{CTr}^*(S),\simeq^{\mathrm{Mor}})$ is always essentially countable but, when $d\geq 2$, it is not essentially treeable.
\end{remark}{}

\section{The space between TSI and CLI}\label{S:final}

With Corollary \ref{Cor:2} we established that the class of CLI Polish groups can produce strictly more complicated orbit equivalence relations than the class of TSI groups from the point of view of Borel (or even Baire-measurable) reductions. The obvious question is how many different complexity classes lie between the class of all classification problems which are classifiable by TSI-group actions and  the ones  which are classifiable by CLI-group actions. In this final section we illustrate how the methods we developed here can be adapted to show that there is an $\omega_1$-sequence of strictly increasing complexity classes. The discussion here will be informal since the details will be provided in an upcoming paper.

Let $\Xspace$ be a Polish $\Gspace$ and recall the unbalanced graph relation $\leftrightsquigarrow$ that we defined between pairs of points of $\Xspace$. In the context of the next definition we may refer to it as the {\bf $1$-unbalanced} relation and we denote it by $\leftrightsquigarrow^1_{\Gspace}$.

\begin{definition}
Let $\Xspace$ be a Polish $\Gspace$-space, let $V$ be an open neighborhood of the identity of $\Gspace$, and let $\alpha<\omega_1$. We define the relation $\leftrightsquigarrow^{\alpha}_{V}$ on $\Xspace$ by induction. Let 
\begin{enumerate}
    \item $x \leftrightsquigarrow^0_{V}y$, if $y\in \overline{V x}$ and $x\in \overline{V y}$;
    \item $x \leftrightsquigarrow^\alpha_{V}y$, if for every open neighborhood $W$ of the identity of $\Gspace$, and every open set $U\subseteq \Xspace$ having non-empty intersection with the orbit of $x$ or $y$, there exist $v^x, v^{y}\in V$ with $v^x x\in U$ and $v^y y\in U$, so that $v^y y    \leftrightsquigarrow^\beta_{W}v^x x$, for all $\beta<\alpha$. 
\end{enumerate}
The {\bf $\alpha$-unbalanced graph} associated to  $\Gspace\curvearrowright\Xspace$  is the graph $(\Xspace/\Gspace,\leftrightsquigarrow^{\alpha}_{\Gspace})$.
\end{definition}

\begin{wrapfigure}{r}{5cm}
\begin{center}

\begin{tikzpicture}[xscale=0.9, yscale =0.5]

\node (b) at (0,0) {};
\node (e) at (2, 10) {CLI};
\draw (b) .. controls (4,7) .. (e);
\draw (b) .. controls (-2,5) .. (e);

\node (t) at (1, 5) {TSI};
\draw (b) .. controls (2,3.5) .. (t);
\draw (b) .. controls (-1,2.5) .. (t);

\node (t) at (1.5, 7.5) {$\alpha$-balanced};
\draw (b) .. controls (3,5.25) .. (t);
\draw (b) .. controls (-1.5,3.75) .. (t);
\end{tikzpicture}
\end{center}
\end{wrapfigure}

In \cite{Malicki2011}, Malicki used iterated Wreath products to define an $\omega_1$-sequence $(P_{\alpha}\mid\alpha<\omega_1)$  of Polish permutation groups. Using this sequence Malicki established that the collection of all CLI group forms a coanalytic non-Borel subset of the standard Borel space of all Polish groups. Using the techniques we developed here one may show that the $\alpha$-unbalanced graph of the Bernoulli shift of $P_{\alpha}$ is generically semi-connected (see Section \ref{Intro_2}) and that any 
orbit equivalence relation with 
generically semi-connected $\beta$-unbalanced graph is generically ergodic for actions of $P_{\alpha}$, when $\alpha<\beta$. 
Moreover, in a certain weak sense, these complexity classes are cofinal in the class of all orbit equivalence relations of CLI groups:   if for any pair $x,y$ of elements of a Polish $\Gspace$-space $\Xspace$ we have $x \leftrightsquigarrow^\alpha_{\Gspace}y$ for all countable ordinals $\alpha$, yet $y \not\in \Gspace x$, then $\Gspace$  cannot be CLI.


\bibliography{main}
\bibliographystyle{alpha}


\end{document}